\documentclass[reqno]{amsart}

\usepackage{amsmath,amssymb,amsthm,amsfonts,enumerate, enumitem,hyperref,cleveref}
\usepackage{dsfont}
\usepackage[all]{xy}
\usepackage[T1]{fontenc}
\usepackage{tikz}
\usepackage{pgfplots}
\pgfplotsset{compat=1.15}
\usepackage{mathrsfs}
\usetikzlibrary{arrows}

\DeclareMathOperator{\Aut}{Aut}

\theoremstyle{plain}

\newtheorem{thrm}{Theorem}[section]
\newtheorem{cor}[thrm]{Corollary}
\newtheorem{prop}[thrm]{Proposition}
\newtheorem{lem}[thrm]{Lemma}

\theoremstyle{definition}

\crefrangeformat{equation}{#3(#1)#4--#5(#2)#6}

\crefname{thrm}{Theorem}{Theorems}
\crefname{theorem}{Theorem}{Theorems}
\crefname{lem}{Lemma}{Lemmas}
\crefname{cor}{Corollary}{Corollaries}
\crefname{prop}{Proposition}{Propositions}
\crefname{defn}{Definition}{Definitions}
\crefname{exm}{Example}{Examples}
\crefname{rem}{Remark}{Remarks}
\crefname{conj}{Conjecture}{Conjectures}
\crefname{quest}{Question}{Questions}
\crefname{section}{Section}{Sections}
\crefname{equation}{\unskip}{\unskip}
\crefname{enumi}{\unskip}{\unskip}
\crefname{subsection}{Subsection}{Subsections}

\newcommand{\e}{\epsilon}

\newcommand{\af}{\alpha}
\newcommand{\bt}{\beta}
\newcommand{\lb}{\lambda}

\newcommand{\gm}{\gamma}
\newcommand{\vf}{\varphi}

\newcommand{\dl}{\delta}

\newcommand{\sg}{\sigma}

\newcommand{\B}{\mathcal{B}}

\newcommand{\m}{{}^{-1}}
\newcommand{\sst}{\subseteq}
\newcommand{\sm}{\setminus}
\newcommand{\impl}{\Rightarrow}

\newcommand{\wht}{\widehat}

\newcommand{\id}{\mathrm{id}}
\renewcommand{\iff}{\Leftrightarrow}

\begin{document}
	\title[Preservers of products equal to primitive idempotents of $I(X,F)$]{Linear maps preserving products equal to primitive idempotents of an incidence algebra}	
	
	\author{Jorge J. Garc{\' e}s}
	\address{Departamento de Matem{\' a}tica Aplicada a la Ingenier{\' i}a Industrial, ETSIDI, Universidad Polit{\' e}cnica de Madrid, Madrid, Spain}
	\email{j.garces@upm.es}
	
	\author{Mykola Khrypchenko}
	\address{Departamento de Matem\'atica, Universidade Federal de Santa Catarina,  Campus Reitor Jo\~ao David Ferreira Lima, Florian\'opolis, SC, CEP: 88040--900, Brazil}
	\email{nskhripchenko@gmail.com}

	\subjclass[2010]{ Primary: 16S50, 15A86; secondary: 17C27}
	\keywords{Incidence algebra; product preserver; primitive idempotent}
	
	\begin{abstract}
		Let $A$, $B$ be algebras and $a\in A$, $b\in B$ a fixed pair of elements. We say that a map $\vf:A\to B$ \textit{preserves products equal to $a$ and $b$} if for all $a_1,a_2\in A$ the equality
		$a_1a_2=a$ implies $\vf(a_1)\vf(a_2)=b$. In this paper we study bijective linear maps $\vf:I(X,F)\to I(X,F)$ preserving products equal to primitive idempotents of $I(X,F)$, where $I(X,F)$ is the incidence algebra of a finite connected poset $X$ over a field $F$. We fully characterize the situation, when such a map $\vf$ exists, and whenever it does, $\vf$ is either an automorphism of $I(X,F)$ or the negative of an automorphism of $I(X,F)$.
	\end{abstract}
	
	\maketitle
	
	\tableofcontents
	
	\section*{Introduction}
	
	There are several ways to generalize the notion of a homomorphism $\vf:A\to B$ between two algebras. One of them deals with a generalization of the equality
	\begin{align}\label{vf(a_1a_2)=vf(a_1)vf(a_2)}
		\vf(a_1a_2)=\vf(a_1)\vf(a_2),
	\end{align}
	leading, for instance, to \textit{Jordan} and \textit{Lie homomorphisms}. The other one consists in a restriction of the set of pairs $(a_1,a_2)$ for which \cref{vf(a_1a_2)=vf(a_1)vf(a_2)} holds. Perhaps, the most known generalizations in this direction are \textit{zero product preservers} ($a_1a_2=0$), \textit{orthogonality preservers} ($a_1a_2=a_2a_1=0$) and \textit{idempotent preservers} ($a_1=a_2\in E(A)$). Inspired by the notion of a zero product preserver, Chebotar, Ke, Lee and Shiao~\cite{Chebotar-Ke-Lee-Shiao05} proposed to study maps that \textit{preserve ``constant products''}, i.e. satisfying
	\begin{align}\label{a_1a_2=a_3a_4=>vf(a_1)vf(a_2)=vf(a_3)vf(a_4)}
		a_1a_2=a_3a_4\impl\vf(a_1)\vf(a_2)=\vf(a_3)\vf(a_4).
	\end{align}
	A variant of this property is the following:
	\begin{align}\label{a_1a_2=a=>vf(a_1)vf(a_2)=b}
		a_1a_2=a\impl\vf(a_1)\vf(a_2)=b,
	\end{align}
	where $a\in A$ and $b\in B$ are fixed. In this case we say that $\vf$ \textit{preserves products equal to $a$ and $b$}. If $a$ and $b$ are zero, we get a zero product preserver. If $a$ and $b$ are the identity elements, we get a map preserving the inverses. 
	
	Chebotar, Ke, Lee and Shiao~\cite{Chebotar-Ke-Lee-Shiao05} proved that a bijective additive map on a division ring satisfying \cref{a_1a_2=a_3a_4=>vf(a_1)vf(a_2)=vf(a_3)vf(a_4)} with $a_1a_2=a_3a_4$ being the identity element is either an automorphism multiplied by a central element or an anti-automorphism multiplied by a central element. This generalizes an old result~\cite[Theorem 1.15]{Artin57} on inverse preservers of division rings. Catalano~\cite{Catalano18} extended the above description by Che\-bo\-tar, Ke, Lee and Shiao to the case, where $a_1a_2$ is an arbitrary invertible element of a division ring. Later, this was generalized in~\cite{Catalano-Hsu-Kapalko19} to maps on the ring $M_n(D)$ of matrices (with usual or Jordan product) over a division ring $D$. The next natural question was to study maps that preserve products equal to non-invertible elements of $M_n(D)$. The case, where such products are rank-one idempotents of $M_n(\mathbb{C})$, was treated by Catalano in~\cite{Catalano21}, and the corresponding maps were proved to be $\pm$~homomorphisms. Bijective maps on $M_n(\mathbb{C})$ preserving products equal to nilpotent matrices were later characterized in~\cite{Catalano-Chang-Lee21} by Catalano and Chang-Lee. Catalano and Julius~\cite{Catalano-Julius21a} solved the problem, where the products are diagonalizable matrices over $\mathbb C$ with the same eigenvalues. Perhaps, the most general situation was considered in~\cite{Catalano-Julius21b}: one of the products is a matrix from $M_n(\mathbb{C})$ of rank at most $n-2$, or both products have the same rank $n-1$. The corresponding product preservers are again of the standard form, i.e. automorphisms multiplied by scalars. When $M_n(\mathbb{C})$ is equipped with the Lie bracket $[a,b]=ab-ba$, the description slightly differs, but still reminds the description of commutativity preservers (i.e. zero product preservers with respect to $[-,-]$), see \cite{Ginsburg-Julius-Velasquez20,Julius21}. In the setting of C$^*$-algebras the problem of studying linear mappings that are $^*$-homomorphisms at a fixed point has been considered in \cite{Burgos-Cabello-Peralta19}.
	
	Notice that the incidence algebra $I(X,F)$ of a finite poset $X$ over a field $F$ can be seen as a subalgebra of the algebra $T_n(F)$ of upper triangular matrices, where $n=|X|$. The primitive idempotents of $I(X,F)$ (which correspond to rank-one idempotent matrices from $T_n(F)$) were described in~\cite{Khripchenko-Novikov09} and the automorphisms of $I(X,F)$ have been studied by various authors (see, for example, \cite{St,Baclawski72,Drozd-Kolesnik07,Kh-aut}), so it is natural to ask whether the Catalano's result~\cite{Catalano21} holds for maps on incidence algebras. Under the assumption that $X$ is connected, we show in \cref{main-result} that a bijective linear map $\vf:I(X,F)\to I(X,F)$ preserving the products equal to fixed primitive idempotents does not always exist, and its existence depends on the existence of an automorphism $\lambda$ of $X$ mapping one fixed element $x\in X$ to another fixed element $y\in X$. If such $\lb\in\Aut(X)$ exists, then $\vf$ is either an automorphism of $I(X,F)$ or the negative of an automorphism of $I(X,F)$. Our result can be applied, in particular, to the upper triangular matrix algebra $T_n(F)$ (see \cref{result-for-T_n(K)}).
	
	\section{Preliminaries}\label{sec-prelim}
	
	\subsection{Rings and algebras}
	An element $a$ of a ring $R$ satisfying $a^2=a$ is called an {\it idempotent}. The set of idempotents of $R$ will be denoted by $E(R)$.
	
	Two elements $a,b\in R$ are {\it orthogonal} if $ab=ba=0$. An idempotent $e\in E(R)$ is {\it primitive} if $e=e_1+e_2$ with $e_1,e_2\in E(R)$ orthogonal implies $e_1,e_2\in\{0,e\}$.

	\subsection{Posets}
	
	A \textit{poset} is a pair $(X,\le)$, where $X$ is a set and $\le$ is a partial order (i.e., a reflexive, transitive and anti-symmetric binary relation) on $X$. All the posets in this article will be finite. The \emph{interval} from $x$ to $y$ in a poset $X$ is the subset $[x,y]:=\{z\in X \mid x\le z\le y\}$. A \textit{chain} in $X$ is a non-empty subset $C\sst X$ such that for all $x,y\in C$ either $x\le y$ or $y\le x$. The \textit{length} of a chain $C\sst X$ is defined to be $|C|-1$. The \textit{length} of a non-empty poset $X$ is $l(X):=\max\{l(C)\mid C\text{ is a chain in }X\}$. We write $l(x,y)$ for $l([x,y])$. A {\it walk} in $X$ is a sequence $x_0,x_1,\dots,x_m$ of elements of $X$, such that for all $i=0,\dots,m-1$ either $x_i\le x_{i+1}$ and $l( x_i,x_{i+1})=1$ or $x_{i+1}\le x_i$ and $l(x_{i+1},x_i)=1$. A poset $X$ is {\it connected} if for all $x,y\in X$ there exists a walk $x=x_0,\dots,x_m=y$. 
	
	Given a subset $A\sst X^2$, denote by $A_\le$ (resp. $A_<$) the set $\{(x,y)\in A\mid x\le y\}$  (resp. $\{(x,y)\in A\mid x<y\}$).
	
	A map $\lb:X\to Y$ between two posets is {\it order-preserving} whenever $x\le y\impl\lb(x)\le\lb(y)$ for all $x,y\in X$. An {\it (order) automorphism} of a finite poset $X$ is an order-preserving bijection of $X$.\footnote{For an infinite $X$ one should also require that the inverse bijection be order-preserving, which holds automatically in the finite case.}
	
	\subsection{Incidence algebras}
	
	The \emph{incidence algebra} $I(X,F)$ of a finite poset $X$ over a field $F$ is the $F$-vector space with basis $\{e_{xy}\mid x\leq y\}$ (called \textit{the standard basis}) and multiplication 
	$$
	e_{xy}e_{uv}=
	\begin{cases}
		e_{xv}, & y=u,\\
		0, & y\ne u.
	\end{cases} 
	$$
	For an element $f\in I(X,F)$ we write $f=\sum_{x\le y}f(x,y)e_{xy}$, so that 
	$$
	(fg)(x,y)=\sum_{x\le z\le y}f(x,z)g(z,y). 
	$$
	The algebra $I(X,F)$ is associative and unital, where the identity element of $I(X,F)$ is $\delta:=\sum_{x\in X} e_{xx}$.
	
	Given $f\in I(X,F)$, define $f_D=\sum_{x\in X}f(x,x)e_{xx}$ and $f_U=\sum_{x<y}f(x,y)e_{xy}$, so that $f=f_D+f_U$. An element $f\in I(X,F)$ is said to be \textit{diagonal} whenever $f=f_D$. The diagonal elements form a commutative subalgebra of $I(X,F)$, denoted by $D(X,F)$. The elements $f$ with $f=f_U$ (equivalently, $f_D=0$) form an ideal which coincides with the {\it Jacobson radical} of $I(X,F)$, denoted $J(I(X,F))$ (see \cite[Theorem 4.2.5]{SpDo}). The set $\B:=\{e_{xy}\mid x<y\}$ is clearly a basis of $J(I(X,F))$. Thus, $I(X,F)=D(X,F)\oplus J(I(X,F))$ as vector spaces. Invertible elements of $I(X,F)$ are exactly those $f\in I(X,F)$ such that $f(x,x)\in F^*$ for all $x\in X$ by \cite[Theorem 1.2.3]{SpDo}. The center of $I(X,F)$ coincides with $\{r\dl\mid r\in F\}$ by \cite[Corollary 1.3.15]{SpDo}, provided that $X$ is connected. For any $Y\sst X$ the incidence algebra $I(Y,F)$ can be seen as a subspace of $I(X,F)$ closed under multiplication, where $Y$ is assumed to be endowed with the induced partial order. Observe that we avoid using the term ``subalgebra'', because the identity element of $I(Y,F)$ is different from that of $I(X,F)$ whenever $Y\ne X$.
	
	We will write $e_{x}:=e_{xx}$ and, more generally, $e_A:=\sum_{a\in A}e_{aa}$ for $A\sst X$. Then $\{e_x\mid x\in X\}$ is a set of pairwise orthogonal primitive idempotents of $I(X,F)$ which form a basis of $D(X,F)$. Observe that $I(X,F)$ also admits a basis formed by the idempotents $\{e_x\}_{x\in X}\sqcup\{e_x+e_{xy}\}_{x<y}$.
	
		Recall from~\cite{St,Baclawski72,Kh-aut} the description of automorphisms of $I(X,F)$. Any order automorphism $\lb$ of $X$ \textit{induces} an automorphism $\widehat\lb$ of $I(X,F)$ by means of 
		\begin{align}\label{hat-lb(e_xy)=e_lb(x)lb(y)}
			\widehat\lb(e_{xy})=e_{\lb(x)\lb(y)},
		\end{align}
		where $x\le y$. Moreover, a map $\sg:X_\le^2\to F^*$ with 
		\begin{align}\label{sg(x_y)sg(y_z)=sg(x_z)}
			\sg(x,y)\sg(y,z)=\sg(x,z)
		\end{align}
		for all $x\le y\le z$ defines the \textit{multiplicative automorphism} $M_\sg$ of $I(X,F)$ acting as follows:
		\begin{align}\label{M_sg(e_xy)=sg(x_y)e_xy}
			M_\sg(e_{xy})=\sg(x,y)e_{xy},
		\end{align}
		where $x\le y$. Any automorphism of $I(X,F)$ is the composition of an inner auto\-mor\-phism of $I(X,F)$ and automorphisms $\widehat\lb$ and $M_\sg$ described in \cref{hat-lb(e_xy)=e_lb(x)lb(y),M_sg(e_xy)=sg(x_y)e_xy}.

	\section{Zero product preservers of $I(X,F)$}
	
	Recall that a map $\vf:A\to B$ between two algebras \textit{preserves zero products} (or is a \textit{zero product preserver}) if $\vf(a)\vf(b)=0$ in $B$ whenever $ab=0$ in $A$. It turns out that the description of zero product preservers of $I(X,F)$ by their action on the natural basis will be useful in the following section.
	
	\begin{prop}\label{zp-on-the-basis}
		Let $X$ be finite and $\vf:I(X,F)\to A$ a linear map, where $A$ is an algebra. Then $\vf$ is a zero product preserver if and only if
		\begin{align}
			\vf(e_{xy})\vf(e_{uv})&=0\text{ for all }x\le y\text{ and }u\le v\text{ with }e_{xy}e_{uv}=0,\label{vf-zp-on-the-basis}\\
			\vf(e_x)\vf(e_{xy})&=\vf(e_{xz})\vf(e_{zy})=\vf(e_{xy})\vf(e_y)\text{ for all }x<z<y.\label{vf(e_xz)_vf(e_zy)=const}
		\end{align}
	\end{prop}
	\begin{proof}
		Let $\vf$ be a zero product preserver. Then \cref{vf-zp-on-the-basis} is immediate. Since $(e_x+e_{xz})(e_{zy}-e_{xy})=0$ for all $x<z\le y$, then using \cref{vf-zp-on-the-basis} we have
		\begin{align*}
			0&=\vf(e_x)\vf(e_{zy})-\vf(e_x)\vf(e_{xy})+\vf(e_{xz})\vf(e_{zy})-\vf(e_{xz})\vf(e_{xy})\\
			&=\vf(e_{xz})\vf(e_{zy})-\vf(e_x)\vf(e_{xy}),
		\end{align*}
		proving \cref{vf(e_xz)_vf(e_zy)=const}.
		
		Conversely, assume \cref{vf-zp-on-the-basis,vf(e_xz)_vf(e_zy)=const}. Take arbitrary $f,g\in I(X,F)$ and calculate
		\begin{align*}
			\vf(f)\vf(g)&=\sum_{x\le y}\sum_{u\le v}f(x,y)g(u,v)\vf(e_{xy})\vf(e_{uv})
			=\sum_{x\le z\le y}f(x,z)g(z,y)\vf(e_{xz})\vf(e_{zy})\\
			&=\sum_{x\le y}\left(\sum_{x\le z\le y}f(x,z)g(z,y)\right)\vf(e_x)\vf(e_{xy})
			=\sum_{x\le y}(fg)(x,y)\vf(e_x)\vf(e_{xy}).
		\end{align*}
		So, if $fg=0$, then $\vf(f)\vf(g)=0$.
	\end{proof}
	
	\section{Preservers of products equal to primitive idempotents of $I(X,F)$}
	
	
	Our goal is to study the bijective linear maps $\vf:I(X,F)\to I(X,F)$ which preserve products equal to a fixed pair of primitive idempotents $\epsilon,\eta\in I(X,F)$, i.e. satisfying \cref{a_1a_2=a=>vf(a_1)vf(a_2)=b} with $a=\epsilon$ and $b=\eta$. By~\cite[Lemma 1]{Khripchenko-Novikov09} there exist $x,y\in X$ and inner automorphisms $\psi_1,\psi_2$ of $I(X,F)$, such that $\psi_1(e_x)=\epsilon$ and $\psi_2(\eta)=e_y$. Then, replacing $\vf$ by $\psi_2\circ\vf\circ\psi_1$, we may assume that $\vf$ satisfies
	\begin{align}\label{fg=e_x=>vf(f)vf(g)=e_y}
		fg=e_x\impl\vf(f)\vf(g)=e_y
	\end{align}
	for all $f,g\in I(X,F)$.
	
	We first list here easy consequences of \cref{fg=e_x=>vf(f)vf(g)=e_y} as in~\cite{Catalano21}.
	\begin{lem}\label{properties-of-vf}
		The following equalities hold:
		\begin{enumerate}
			\item $\vf(e_x)^2=e_y$;\label{vf(e_x)^2=e_y}
			\item 
			\begin{enumerate}
				\item $\vf(e_x)\vf(e_{uv})=0$ for all $x\ne u\le v$;\label{vf(e_x)vf(e_uv)=0}
				\item $\vf(e_{uv})\vf(e_x)=0$ for all $u\le v\ne x$;\label{vf(e_uv)vf(e_x)=0} 
			\end{enumerate}
			\item 
			\begin{enumerate}
				\item $\vf(e_z)\vf(e_{ux})=0$ for all $z\ne u\le x$;\label{vf(e_z)vf(e_ux)=0}
				\item $\vf(e_{xv})\vf(e_z)=0$ for all $x\le v\ne z$;\label{vf(e_xv)vf(e_z)=0}
			\end{enumerate}
			\item $\vf(e_u)\vf(e_v)=0$ for all $u\ne v$;\label{vf(e_u)vf(e_v)=0}
			\item 
			\begin{enumerate}
				\item $\vf(e_{ux})\vf(e_x)-\vf(e_u)\vf(e_{ux})=r\vf(e_{ux})^2$ for all $u<x$ and $r\in F^*$; \label{r(vf(e_ux)vf(e_x)-vf(e_u)vf(e_ux))-r^2vf(e_ux)^2=0} 
				\item $\vf(e_x)\vf(e_{xv})-\vf(e_{xv})\vf(e_v)=r\vf(e_{xv})^2$ for all $v>x$ and $r\in F^*$; \label{r(vf(e_x)vf(e_xv)-vf(e_xv)vf(e_v))-r^2vf(e_xv)^2=0} 
			\end{enumerate}
			\item 
			\begin{enumerate}
				\item $\vf(e_{ux})\vf(e_z)=0$ for all $u\le x$, $z\not\in\{u,x\}$;\label{vf(e_ux)vf(e_z)=0}
				\item $\vf(e_z)\vf(e_{xv})=0$ for all $x\le v$, $z\not\in\{x,v\}$.\label{vf(e_z)vf(e_xv)=0}
			\end{enumerate}
			
			
		\end{enumerate}
	\end{lem}
	\begin{proof}
		\cref{vf(e_x)^2=e_y} follows from $e_x^2=e_x$; \cref{vf(e_x)vf(e_uv)=0} follows from \cref{vf(e_x)^2=e_y} and $e_x(e_x+e_{uv})=e_x$ for $x\ne u\le v$; \cref{vf(e_uv)vf(e_x)=0} follows from \cref{vf(e_x)^2=e_y} and $(e_x+e_{uv})e_x=e_x$ for $u\le v\ne x$. 
		
		It suffices to prove \cref{vf(e_z)vf(e_ux)=0} for $u<x$, since $u=x$ is a particular case of \cref{vf(e_uv)vf(e_x)=0}. So, assume $u<x$ and write $\vf(e_x)\vf(e_{ux})=0$ by \cref{vf(e_x)vf(e_uv)=0}. In particular, this proves \cref{vf(e_z)vf(e_ux)=0} for $z=x$. Thus, take $z\ne x$. Then $(e_x+e_z)(e_x+e_{ux})=e_x$. Combining this with $\vf(e_z)\vf(e_x)=0$ (by \cref{vf(e_uv)vf(e_x)=0}) and \cref{vf(e_x)^2=e_y}, we obtain \cref{vf(e_z)vf(e_ux)=0}. Similarly, \cref{vf(e_xv)vf(e_z)=0} follows from $\vf(e_{xv})\vf(e_x)=0$ for $x<v$, $(e_x+e_{xv})(e_x+e_z)=e_x$ for $z\ne x$ and \cref{vf(e_x)^2=e_y}.
		
		If $x\in\{u,v\}$, then \cref{vf(e_u)vf(e_v)=0} is a particular case of \cref{vf(e_x)vf(e_uv)=0} or \cref{vf(e_uv)vf(e_x)=0}. Otherwise, consider the product $(e_x+e_u)(e_x+e_v)=e_x$ and apply \cref{vf(e_x)^2=e_y,vf(e_x)vf(e_uv)=0,vf(e_uv)vf(e_x)=0}.
		
		For \cref{r(vf(e_ux)vf(e_x)-vf(e_u)vf(e_ux))-r^2vf(e_ux)^2=0} write $(e_x+e_u+re_{ux})(e_x-re_{ux})=e_x$ and use \cref{vf(e_x)^2=e_y,vf(e_x)vf(e_uv)=0,vf(e_uv)vf(e_x)=0}. 
		Item \cref{r(vf(e_x)vf(e_xv)-vf(e_xv)vf(e_v))-r^2vf(e_xv)^2=0} is similar and follows from $(e_x-re_{xv})(e_x+e_v+re_{xv})=e_x$.
		
		For \cref{vf(e_ux)vf(e_z)=0} write $(e_x+e_u+e_{ux})(e_x-e_{ux}+e_z)=e_x$ and use \cref{r(vf(e_ux)vf(e_x)-vf(e_u)vf(e_ux))-r^2vf(e_ux)^2=0,vf(e_x)vf(e_uv)=0,vf(e_u)vf(e_v)=0}. Analogously,
		\cref{vf(e_z)vf(e_xv)=0} is a consequence of \cref{r(vf(e_x)vf(e_xv)-vf(e_xv)vf(e_v))-r^2vf(e_xv)^2=0,vf(e_uv)vf(e_x)=0,vf(e_u)vf(e_v)=0}.
	\end{proof}
	
	\begin{cor}\label{vf(e_ux)(z_z)-and-vf(e_xv)(z_z)}
		Let $z\in X\setminus\{y\}$.
		\begin{enumerate}
			\item If $\vf(e_u)(z,z)=0$ for some $u<x$, then $\vf(e_{ux})(z,z)=0$.\label{vf(e_ux)(z_z)=0}
			\item If $\vf(e_v)(z,z)=0$ for some $v>x$, then $\vf(e_{xv})(z,z)=0$.\label{vf(e_xv)(z_z)=0}
		\end{enumerate}
	\end{cor}
	\begin{proof}
		Let us prove \cref{vf(e_ux)(z_z)=0}. Using \cref{properties-of-vf}\cref{vf(e_x)^2=e_y} we have $\vf(e_x)(z,z)^2=e_y(z,z)=0$, so $\vf(e_x)(z,z)=0$. Therefore, 
		\begin{align*}
			r\vf(e_{ux})(z,z)^2=\vf(e_{ux})(z,z)\vf(e_x)(z,z)-\vf(e_u)(z,z)\vf(e_{ux})(z,z)=0
		\end{align*}
		by \cref{properties-of-vf}\cref{r(vf(e_ux)vf(e_x)-vf(e_u)vf(e_ux))-r^2vf(e_ux)^2=0}, whence $\vf(e_{ux})(z,z)=0$. The proof of \cref{vf(e_xv)(z_z)=0} similarly follows from \cref{properties-of-vf}\cref{r(vf(e_x)vf(e_xv)-vf(e_xv)vf(e_v))-r^2vf(e_xv)^2=0}.
	\end{proof}
	
	From now on we view $I(X\setminus\{x\},F)$ and $I(X\setminus\{y\},F)$ as subspaces of $I(X,F)$ closed under multiplication. Their identity elements are $e_{X\sm\{x\}}$ and $e_{X\sm\{y\}}$, respectively.
	
	\begin{lem}
		Let $f\in I(X,F)$. If $f^2=e_y$, then $f=\pm e_y+g$ for some $g\in J(I(X\setminus\{y\},F))$ with $g^2=0$.
	\end{lem}
	\begin{proof}
		It immediately follows that $f_D=\pm e_y$. Now take $u<y$ and write
		\begin{align*}
			0&=e_y(u,y)=f^2(u,y)=f(u,y)f(y,y)+\sum_{u<v<y}f(u,v)f(v,y)\\
			&=\pm f(u,y)+\sum_{u<v<y}f(u,v)f(v,y),
		\end{align*}
		whence
		\begin{align}\label{f(u_y)=sum_f(u_v)f(v_y)}
			f(u,y)=\mp\sum_{u<v<y}f(u,v)f(v,y).
		\end{align}
		If $l(u,y)=1$, then there is no $v\in X$ with $u<v<y$, so the sum on the right-hand side of \cref{f(u_y)=sum_f(u_v)f(v_y)} vanishes, and $f(u,y)=0$. If $l(u,y)>1$, then $l(v,y)<l(u,y)$ for all $u<v<y$, so the obvious induction argument implies $f(u,y)=0$.
		
		Similarly, for all $y<v$ one has
		\begin{align*}
			0&=e_y(y,v)=f^2(y,v)=f(y,y)f(y,v)+\sum_{y<u<v}f(y,u)f(u,v)\\
			&=\pm f(y,v)+\sum_{y<u<v}f(y,u)f(u,v),
		\end{align*}
		whence $f(y,v)=0$ by induction on $l(y,v)$.
		
			Thus, $f=\pm e_y+g$, where $g=f_U$. Observe that $e_yg=\sum_{y<v}f(y,v)e_{yv}=0$ and $ge_y=\sum_{u<y}f(u,y)e_{uy}=0$. Therefore, $e_y=f^2=e_y\pm e_yg\pm ge_y+g^2=e_y+g^2$, whence $g^2=0$.
	\end{proof}
	
	\begin{cor}\label{vf(e_x)=pm.e_y+g}
		We have $\vf(e_x)=\pm e_y+g$ for some $g\in J(I(X\setminus\{y\},F))$ with $g^2=0$.
	\end{cor}

	\begin{lem}\label{vf(I(X-minus-x))-sst-I(X-minus-y)}
		We have $\vf(I(X\setminus\{x\},F))\sst I(X\setminus\{y\},F)$. Moreover, $\vf$ is a zero product preserver from $I(X\setminus\{x\},F)$ to $I(X\setminus\{y\},F)$.
	\end{lem}
	\begin{proof}
		Let $u,v\in X\setminus\{x\}$, $u\le v$. Then $\vf(e_x)\vf(e_{uv})=0$ by \cref{properties-of-vf}\cref{vf(e_x)vf(e_uv)=0}. Multiplying this by $\vf(e_x)$ on the left and using \cref{properties-of-vf}\cref{vf(e_x)^2=e_y} we obtain $e_y\vf(e_{uv})=0$. Similarly $\vf(e_{uv})e_y=0$ by \cref{vf(e_x)^2=e_y,vf(e_uv)vf(e_x)=0} of \cref{properties-of-vf}. Therefore, $\vf(e_{uv})(a,y)=\vf(e_{uv})(y,b)$ for all $a\le y\le b$.  This proves the desired inclusion.
		
		Now take $f,g\in I(X\setminus\{x\},F)$ such that $fg=0$. It is easily seen that $(e_x+f)(e_x+g)=e_x$ in $I(X,F)$. Hence, $(\vf(e_x)+\vf(f))(\vf(e_x)+\vf(g))=e_y$ thanks to \cref{fg=e_x=>vf(f)vf(g)=e_y}. The latter is equivalent to $\vf(f)\vf(g)=0$ by \cref{vf(e_x)^2=e_y,vf(e_x)vf(e_uv)=0,vf(e_uv)vf(e_x)=0} of \cref{properties-of-vf}.
	\end{proof}
	
	Observe that $\dim(I(X\setminus\{x\},F))=|(X\setminus\{x\})^2_\le|$ and $\dim(I(X\setminus\{y\},F))=|(X\setminus\{y\})^2_\le|$, so we have the following immediate consequence of \cref{vf(I(X-minus-x))-sst-I(X-minus-y)}.
	\begin{cor}\label{|(X-minus-x)^2_le|>|(X-minus-y)^2_le|}
		If $|(X\setminus\{x\})^2_\le|>|(X\setminus\{y\})^2_\le|$, then there is no bijective linear map $\vf:I(X,F)\to I(X,F)$ satisfying \cref{fg=e_x=>vf(f)vf(g)=e_y}.
	\end{cor}
	
	Thus, in what follows we assume $|(X\setminus\{x\})^2_\le|\le|(X\setminus\{y\})^2_\le|$.
	
	\begin{cor}\label{vf(J(I(X-minus-x)))-sst-J(I(X-minus-y))}
		We have $\vf(J(I(X\setminus\{x\},F)))\sst J(I(X\setminus\{y\},F))$.
	\end{cor}
	\begin{proof}
		Indeed, for any $u,v\in X\setminus\{x\}$, $u<v$, it follows from $e_{uv}^2=0$ and \cref{vf(I(X-minus-x))-sst-I(X-minus-y)} that $\vf(e_{uv})^2=0$, so $\vf(e_{uv})\in J(I(X\setminus\{y\},F))$. 
	\end{proof}
	
	\begin{lem}\label{vf(e_X-minus-x)-invertible}
		The element $\vf\left(e_{X\setminus\{x\}}\right)$ is invertible in $I(X\setminus\{y\},F)$.
	\end{lem}
	\begin{proof}
		For all $u\in X\setminus\{x\}$ denote $Z_u=\{z\in X\mid \vf(e_u)(z,z)\ne 0\}\sst X\setminus\{y\}$. Observe that $Z_u\cap Z_v=\emptyset$ for $u\ne v$ thanks to \cref{properties-of-vf}\cref{vf(e_u)vf(e_v)=0}. Since $e_{X\setminus\{x\}}=\sum_{u\in X\setminus\{x\}}e_u$, we have
		\begin{align*}
			Z:=\left\{z\in X\mid \vf\left(e_{X\setminus\{x\}}\right)(z,z)\ne 0\right\}=\bigsqcup_{u\in X\setminus\{x\}} Z_u.
		\end{align*}
		Assume that $\vf\left(e_{X\setminus\{x\}}\right)\not\in I(X\setminus\{y\},F)^*$. Then $Z\ne X\setminus\{y\}$, so there exists $z_0\in X\setminus\{y\}$ such that $z_0\not\in Z$. It follows that $\vf(e_u)(z_0,z_0)=0$ for all $u\in X\setminus\{x\}$. By \cref{vf(e_ux)(z_z)-and-vf(e_xv)(z_z)} we conclude that $\vf(e_{ux})(z_0,z_0)=0$ for all $u<x$ and $\vf(e_{xv})(z_0,z_0)=0$ for all $v>x$. Furthermore, $\vf(e_{uv})(z_0,z_0)=0$ for all $u,v\in X\setminus\{x\}$, $u<v$, by \cref{vf(J(I(X-minus-x)))-sst-J(I(X-minus-y))}. Finally, $\vf(e_x)(z_0,z_0)=0$ by \cref{vf(e_x)=pm.e_y+g}. Thus, $\vf(f)(z_0,z_0)=0$ for all $f\in I(X,F)$, a contradiction with $e_{z_0}\in\vf(I(X,F))$.
	\end{proof}

	

	\begin{lem}\label{vf=c.psi}
		We have $\vf(e_x)=\pm e_y$ and there exist a monomorphism $\psi:I(X\setminus\{x\},F)\to I(X\setminus\{y\},F)$ and $c\in I(X\setminus\{y\},F)^*$ such that $\vf(f)=c\psi(f)$ for all $f\in I(X\setminus\{x\},F)$.
	\end{lem}
	\begin{proof}
		Recall that $I(X\setminus\{x\},F)$ has a basis formed by idempotents. Therefore, since $\vf:I(X\sm\{x\},F)\to I(X\sm\{y\},F)$ preserves zero products by \cref{vf(I(X-minus-x))-sst-I(X-minus-y)} and $\vf(e_{X\setminus\{x\}})$ is invertible in $I(X\sm\{y\},F)$, it follows from \cite[Theorem 2.6(vi)]{Chebotar-Ke-Lee-Wong03} that there are a homomorphism $\psi:I(X\setminus\{x\},F)\to I(X\setminus\{y\},F)$ and $c=\vf(e_{X\setminus\{x\}})\in I(X\setminus\{y\},F)^*$ such that $\vf(f)=c\psi(f)$ for all $f\in I(X\setminus\{x\},F)$. The homomorphism $\psi$ is injective, because $\vf$ is injective and $c$ is invertible. Recall from \cref{vf(e_x)=pm.e_y+g} that $\vf(e_x)=\pm e_y+g$, where $g\in I(X\setminus\{y\},F)$. Now, $\vf(e_x)c=\vf(e_x)\vf(e_{X\setminus\{x\}})=0$ by \cref{properties-of-vf}\cref{vf(e_u)vf(e_v)=0}. On the other hand, $\vf(e_x)c=\pm e_yc+gc=gc$. Hence, $gc=0$. But $c$ is invertible in $I(X\setminus\{y\},F)$, so $g=0$. Thus, $\vf(e_x)=\pm e_y$.
	\end{proof}
	
	\begin{lem}\label{lem conjugate T} 
		There exist $\beta\in I(X,F)^*$, $\{\af_z\}_{z\in X}\sst F^*$ and a bijection $\lb:X\to X$ with $\lb(x)=y$ such that 
		\begin{align}\label{bt.vf(e_z).bt-inv=af_z.e_lb(z)}
			\beta \vf(e_z) \beta^{-1}=\af_z e_{\lb(z)}
		\end{align}
		for all $z\in X$. Moreover, 
		\begin{align}\label{bt.e_y.bt-inv=e_y}
			\bt e_y\bt\m=e_y   
		\end{align}
		and $\af_x=1\iff \vf(e_x)=e_y$, $\af_x=-1\iff \vf(e_x)=-e_y$.
	\end{lem}
	
	\begin{proof}
		By \cref{vf=c.psi} for any $z\in X\setminus\{x\}$ we have $\vf(e_z)=c\psi(e_z)$, where $c\in I(X\setminus\{y\},F)^*$ and $\psi$ is a monomorphism $\psi:I(X\setminus\{x\},F)\to I(X\setminus\{y\},F)$. Since $\{\psi(e_z)\}_{z\in X\setminus\{x\}}$ is a set of orthogonal idempotents and $|X\setminus\{x\}|=|X\setminus\{y\}|$, there exists a bijection $\lb':X\setminus\{x\}\to X\setminus\{y\}$ such that $\psi(e_z)_D=e_{\lb'(z)}$. Hence, by \cite[Lemma 5.4]{GK} there is $\beta'\in I(X\setminus\{y\},F)^*$ such that $\bt'\psi(e_z)\bt'\m= e_{\lb'(z)}$ for all $z\in X\setminus\{x\}$. Then 
		\begin{align}\label{bt'.vf(e_z).bt-inv=gm.e_lb'(z)}
			\bt'\vf(e_z)\bt'\m=\bt'c\psi(e_z)\bt'\m =\bt'c\bt'\m e_{\lb'(z)}.
		\end{align}
		Denote $\gm=\bt'c\bt'\m\in I(X\setminus\{y\},F)^*$. Let $u<v$ in $X\sm\{y\}$ and $a,b\in X\sm\{x\}$ such that $\lb'(a)=u$ and $\lb'(b)=v$. Then it follows from \cref{properties-of-vf}\cref{vf(e_u)vf(e_v)=0} and \cref{bt'.vf(e_z).bt-inv=gm.e_lb'(z)} that
		\begin{align*}
			0=\vf(e_a)\vf(e_b)=\gm e_u\gm e_v=\gm\cdot\gm(u,v) e_{uv}.
		\end{align*}
		Since $\gm$ is invertible in $I(X\setminus\{y\},F)$, then $\gm(u,v)e_{uv}=0$, whence $\gm(u,v)=0$. Thus, $\gm$ is diagonal and we define $\af_z=\gm(z,z)\in F^*$ for all $z\in X\sm\{x\}$, so that 
		\begin{align}\label{bt'.vf(e_z).bt-inv=af_z.e_lb'(z)}
			\bt'\vf(e_z)\bt'\m=\af_z e_{\lb'(z)}
		\end{align}
		for such $z$.
		
		We now extend $\lb'$ to a bijection $\lb:X\to X$ by $\lb(x)=y$ and define $\bt=e_y+\bt'$, 
		\begin{align*}
			\af_x=\begin{cases}
				1, & \vf(e_x)=e_y,\\
				-1, & \vf(e_x)=-e_y.
			\end{cases}
		\end{align*}
		Then $\bt\in I(X,F)^*$ and using \cref{bt'.vf(e_z).bt-inv=af_z.e_lb'(z)} we obtain
		\begin{align*}
			\bt \vf(e_z) \bt\m=(e_y+\bt')\vf(e_z)(e_y+\bt'\m)=\bt'\vf(e_z)\bt'\m=\af_z e_{\lb'(z)}
			=\af_z e_{\lb(z)}
		\end{align*}
		for all $z\in X\sm\{x\}$ and
		\begin{align*} 
			\bt \vf(e_x) \bt\m=(e_y+\bt')(\pm e_y)(e_y+\bt'\m)=\pm e_y=\af_x e_{\lb(x)}.
		\end{align*}
		So, \cref{bt.vf(e_z).bt-inv=af_z.e_lb(z)} holds for all $z\in X$. Moreover,
		\begin{align*}
			\bt e_y\bt\m=(e_y+\bt')e_y(e_y+\bt'\m)=e_y,
		\end{align*}
		proving \cref{bt.e_y.bt-inv=e_y}.
	\end{proof}
	
	Observe from \cref{lem conjugate T,fg=e_x=>vf(f)vf(g)=e_y,bt.vf(e_z).bt-inv=af_z.e_lb(z)} that if $f g=e_x$ then 
	\begin{align*}
		\beta \vf(f) \beta^{-1}\beta \vf(g) \beta^{-1}=\beta \vf(f)\vf(g) \beta^{-1}=\beta e_y \beta^{-1}=e_y,   
	\end{align*}
	so, without loss of generality, we may assume that, for all $z\in X$,
	\begin{align}\label{vf(e_z)=af_ze_lb(z)}
		\vf(e_z)=\af_z e_{\lb(z)},
	\end{align}
	where $\af_z\in F^*$, $\af_x=1\iff \vf(e_x)=e_y$, $\af_x=-1\iff \vf(e_x)=-e_y$ and $\lb:X\to X$ is a bijection with $\lb(x)=y$.  
	
	\begin{lem}\label{lambda-auto}
		The bijection $\lb$ is an automorphism of $X$.
	\end{lem}
	\begin{proof}
		Let $u,v\in X\setminus\{x\}$ with $u\le v$. Then $e_u,e_v,e_{uv}\in I(X\setminus\{x\},F)$. By \cref{vf=c.psi,vf(e_z)=af_ze_lb(z)}
		\begin{align*}
			0\ne \vf(e_{uv})=c\psi(e_{uv})=c\psi(e_ue_{uv}e_v)=c\psi(e_u)\psi(e_{uv})\psi(e_v)=ce_{\lb(u)}\psi(e_{uv})e_{\lb(v)},
		\end{align*}
		whence $\lb(u)\le\lb(v)$. So, $\lb|_{X\setminus\{x\}}$ is order-preserving.
		
		Now let $u<x$. By \cref{properties-of-vf}\cref{vf(e_z)vf(e_ux)=0} and \cref{vf(e_z)=af_ze_lb(z)} for any $z\ne u$ we have
		\begin{align*}
			0=\vf(e_z)\vf(e_{ux})=\af_z e_{\lb(z)}\vf(e_{ux}),
		\end{align*}
		so $e_{\lb(z)}\vf(e_{ux})=0$ for all $z\ne u$. It follows that 
		\begin{align}\label{vf(e_ux)=e_lb(u)vf(e_ux)}
			\vf(e_{ux})=\dl\cdot\vf(e_{ux})=\left(\sum_{z\in X}e_{\lb(z)}\right)\vf(e_{ux})=e_{\lb(u)}\vf(e_{ux}).
		\end{align}
		
		Now by \cref{vf(e_z)=af_ze_lb(z),vf(e_ux)=e_lb(u)vf(e_ux)} and \cref{properties-of-vf}\cref{r(vf(e_ux)vf(e_x)-vf(e_u)vf(e_ux))-r^2vf(e_ux)^2=0} (with $r=1$) we have
		\begin{align}
			0&=\vf(e_{ux})\af_x e_{\lb(x)}-\af_u e_{\lb(u)}\vf(e_{ux})-e_{\lb(u)}\vf(e_{ux})e_{\lb(u)}\vf(e_{ux})\notag\\
			&=e_{\lb(u)}\vf(e_{ux})\af_x e_{\lb(x)}-\af_u\vf(e_{ux})-\vf(e_{ux})(\lb(u),\lb(u))e_{\lb(u)}\vf(e_{ux})\notag\\
			&=\af_xe_{\lb(u)}\vf(e_{ux}) e_{\lb(x)}-(\af_u+\vf(e_{ux})(\lb(u),\lb(u)))\vf(e_{ux}).\label{0=af_xe_lb(u)vf(e_ux)e_lb(x)-(af_u+vf(e_ux)(lb(u)lb(u)))vf(e_ux)}
		\end{align}
		On the other hand, by \cref{vf(e_z)=af_ze_lb(z)} and \cref{properties-of-vf}\cref{vf(e_ux)vf(e_z)=0} we have $\vf(e_{ux})e_{\lb(z)}=0$ for all $z\not\in\{u,x\}$, so combining this with \cref{vf(e_ux)=e_lb(u)vf(e_ux)} we get
		\begin{align}
			\vf(e_{ux})&=\vf(e_{ux})\cdot\dl=\vf(e_{ux})\cdot\left(\sum_{z\in X}e_{\lb(z)}\right)=\vf(e_{ux})e_{\lb(u)}+\vf(e_{ux})e_{\lb(x)}\notag\\
			&=e_{\lb(u)}\vf(e_{ux})e_{\lb(u)}+e_{\lb(u)}\vf(e_{ux})e_{\lb(x)}\notag\\
			&=\vf(e_{ux})(\lb(u),\lb(u))e_{\lb(u)}+e_{\lb(u)}\vf(e_{ux})e_{\lb(x)}.\label{vf(e_ux)=vf(e_ux)(lb(u)lb(u))e_lb(u)+e_lb(u)vf(e_ux)e_lb(x)}
		\end{align}
		Suppose that $e_{\lb(u)}\vf(e_{ux})e_{\lb(x)}=0$. Then thanks to \cref{0=af_xe_lb(u)vf(e_ux)e_lb(x)-(af_u+vf(e_ux)(lb(u)lb(u)))vf(e_ux)} and $\vf(e_{ux})\ne 0$ we obtain $\vf(e_{ux})(\lb(u),\lb(u))=-\af_u$. Therefore, \cref{vf(e_ux)=vf(e_ux)(lb(u)lb(u))e_lb(u)+e_lb(u)vf(e_ux)e_lb(x),vf(e_z)=af_ze_lb(z)} give $\vf(e_{ux})=-\af_ue_{\lb(u)}=-\vf(e_u)=\vf(-e_u)$, whence $e_{ux}=-e_u$, a contradiction. Thus, $e_{\lb(u)}\vf(e_{ux})e_{\lb(x)}\ne 0$, i.e. $\lb(u)<\lb(x)$ and $\vf(e_{ux})(\lb(u),\lb(x))\ne 0$.
		
		Similarly one proves that $\lb(x)<\lb(v)$ for all $x<v$.
		
		Thus, $\lb$ is order-preserving. Since $X$ is finite, $\lb\m$ is also order-preserving. Thus, $\lb$ is an automorphism of $X$.
	\end{proof}

	\begin{lem}\label{af_u=af_x-and-af_v=af_x}
		We have $\af_u=\af_x$ and
		\begin{align}\label{vf(e_ux)=vf(e_ux)(lb(u)lb(x))e_lb(u)lb(x)}
			\vf(e_{ux})=\vf(e_{ux})(\lb(u),\lb(x))e_{\lb(u)\lb(x)}
		\end{align}
		for all $u<x$. Similarly, $\af_v=\af_x$ and
		\begin{align}\label{vf(e_xv)=vf(e_xv)(lb(x)lb(v))e_lb(x)lb(v)}
			\vf(e_{xv})=\vf(e_{xv})(\lb(x),\lb(v))e_{\lb(x)\lb(v)}
		\end{align}
		for all $v>x$.
	\end{lem}
	\begin{proof}
		Substituting the right-hand side of \cref{vf(e_ux)=vf(e_ux)(lb(u)lb(u))e_lb(u)+e_lb(u)vf(e_ux)e_lb(x)} to \cref{0=af_xe_lb(u)vf(e_ux)e_lb(x)-(af_u+vf(e_ux)(lb(u)lb(u)))vf(e_ux)}, we obtain
		\begin{align*}
			0&=(\af_x-\af_u-\vf(e_{ux})(\lb(u),\lb(u)))\vf(e_{ux})(\lb(u),\lb(x))e_{\lb(u)\lb(x)}\\
			&\quad-(\af_u+\vf(e_{ux})(\lb(u),\lb(u)))\vf(e_{ux})(\lb(u),\lb(u))e_{\lb(u)}.
		\end{align*}
		If $\vf(e_{ux})(\lb(u),\lb(u))\ne 0$, then $\af_u+\vf(e_{ux})(\lb(u),\lb(u))=0$, so $\vf(e_{ux})(\lb(u),\lb(x))=0$, a contradiction with $e_{\lb(u)}\vf(e_{ux})e_{\lb(x)}\ne 0$ proved in \cref{lambda-auto}. Consequently, $\vf(e_{ux})(\lb(u),\lb(u))=0$ and $\af_u=\af_x$. Then \cref{vf(e_ux)=vf(e_ux)(lb(u)lb(x))e_lb(u)lb(x)} follows from \cref{vf(e_ux)=vf(e_ux)(lb(u)lb(u))e_lb(u)+e_lb(u)vf(e_ux)e_lb(x)}. The proof of $\af_v=\af_x$ and \cref{vf(e_xv)=vf(e_xv)(lb(x)lb(v))e_lb(x)lb(v)} for all $v>x$ is similar.
	\end{proof}  
	
	\begin{cor}\label{vf(J(I(X_F)))=J(I(X_F))}
		We have $\vf(J(I(X,F))=J(I(X,F))$.
	\end{cor}
	\begin{proof}
		Observe that $\vf(e_{ux}),\vf(e_{xv})\in J(I(X,F))$ for all $u<x<v$ by \cref{af_u=af_x-and-af_v=af_x}. Together with \cref{vf(J(I(X-minus-x)))-sst-J(I(X-minus-y))} this results in $\vf(J(I(X,F)))\sst J(I(X,F))$. Since $\vf$ is injective, we have $\vf(J(I(X,F)))=J(I(X,F))$.
	\end{proof}
	
	Let $\widehat\lb$ be the automorphism of $I(X,F)$ induced by $\lb$ as in \cref{hat-lb(e_xy)=e_lb(x)lb(y)}. Replacing $\vf$ by $(\widehat\lb)\m\circ\vf$, we may assume that $x=y$ and, for all $z\in X$,
	\begin{align}\label{vf(e_z)=gm.e_z}
		\vf(e_z)=\af_z e_z,
	\end{align}
	where $\af_z\in F^*$ with $\alpha_x=1\iff \vf(e_x)=e_x$ and $\alpha_x=-1\iff \vf(e_x)=-e_x$.  
	
	\begin{lem}\label{vf(e_uv)-multiple-e_uv}
		For all $u<v$ in $X$ we have 
		\begin{align}\label{vf(e_uv)=sg(u_v)e_uv}
			\vf(e_{uv})=\sg(u,v) e_{uv}
		\end{align}
		for some $\sg(u,v)\in F^*$.
	\end{lem}
	\begin{proof}
		It suffices to prove that $\vf(e_{uv})=e_u\vf(e_{uv})=\vf(e_{uv})e_v$ for all $u<v$, because in this case
		\begin{align*}
			\vf(e_{uv})=e_u\vf(e_{uv})=e_u\vf(e_{uv})e_v=\vf(e_{uv})(u,v)e_{uv},
		\end{align*}
		so we define $\sg(u,v)=\vf(e_{uv})(u,v)\in F^*$.
		
		\textit{Case 1.} $u,v\in X\sm\{x\}$. Then $\vf(e_x)\vf(e_{uv})=0$ by \cref{properties-of-vf}\cref{vf(e_x)vf(e_uv)=0}, which implies $e_x\vf(e_{uv})=0$ by \cref{vf(e_z)=gm.e_z}. Similarly, $\vf(e_z)\vf(e_{uv})=0$ for all $z\ne u,x$ by \cref{vf(I(X-minus-x))-sst-I(X-minus-y)}, so $e_z\vf(e_{uv})=0$ for all $z\ne u,x$. It follows that
		\begin{align*}
			\vf(e_{uv})=\dl\cdot\vf(e_{uv})=\left(\sum_{z\in X}e_z\right)\vf(e_{uv})=e_u\vf(e_{uv}).
		\end{align*}
		In a similar way one proves using \cref{properties-of-vf}\cref{vf(e_uv)vf(e_x)=0}, \cref{vf(I(X-minus-x))-sst-I(X-minus-y),vf(e_z)=gm.e_z} that $\vf(e_{uv})=\vf(e_{uv})e_v$, 
		
		\textit{Case 2.} $x=u<v$. Then $\vf(e_{uv})\vf(e_z)=\vf(e_{xv})\vf(e_z)=0$ for all $z\ne v$ by \cref{properties-of-vf}\cref{vf(e_xv)vf(e_z)=0}, whence $\vf(e_{uv})e_z=0$ for all $z\ne v$ in view of \cref{vf(e_z)=gm.e_z}. Therefore, 
		\begin{align}\label{vf(e_uv)=vf(e_uv)e_v}
			\vf(e_{uv})=\vf(e_{uv})\cdot \dl=\vf(e_{uv})e_v.
		\end{align}
		Now, $\vf(e_z)\vf(e_{uv})=\vf(e_z)\vf(e_{xv})=0$ for $z\ne x,v$ by \cref{properties-of-vf}\cref{vf(e_z)vf(e_xv)=0}, so $e_z\vf(e_{uv})=0$ for $z\ne u,v$ by \cref{vf(e_z)=gm.e_z}. Consequently, using \cref{vf(e_uv)=vf(e_uv)e_v,vf(J(I(X_F)))=J(I(X_F))}, we get
		\begin{align*}
			\vf(e_{uv})=\dl\cdot\vf(e_{uv})=e_u\vf(e_{uv})+e_v\vf(e_{uv})=e_u\vf(e_{uv})+e_v\vf(e_{uv})e_v=e_u\vf(e_{uv}),
		\end{align*}
		as needed.
		
		\textit{Case 3.} $u<v=x$. This case is similar to Case 2. 
	\end{proof}
	
	\begin{lem}\label{lem vf(1) central}
		Let $X$ be connected. Then $\vf(\dl)=\pm \delta$.
	\end{lem}
	\begin{proof}
		Using \cref{vf(e_z)=gm.e_z} we write
		\begin{align*}
			\vf(\dl)=\sum_{z\in X}\vf(e_z)=\sum_{z\in X}\af_z e_z.
		\end{align*}
		It suffices to show that $\af_z=\af_x$ for all $z\in X$. We know by \cref{af_u=af_x-and-af_v=af_x} that $\af_u=\af_x$ for all $u<x$ and $\af_v=\af_x$ for all $v>x$. Moreover, for all $u<v$, in view of \cref{vf(e_uv)-multiple-e_uv,vf(e_z)=gm.e_z} we obtain
		\begin{align*}
			\af_u\vf(e_{uv})=\sg(u,v)\af_u e_{uv}=\sg(u,v)\af_u e_u\cdot e_{uv}=\sg(u,v)\vf(e_u)\cdot e_{uv}=\vf(e_u)\vf(e_{uv}).
		\end{align*}
		Similarly,
		\begin{align*}
			\af_v\vf(e_{uv})=\sg(u,v) \af_v e_{uv}=\sg(u,v)e_{uv}\cdot \af_v e_v=\sg(u,v)e_{uv}\vf(e_v)=\vf(e_{uv})\vf(e_v).
		\end{align*}
		However, $\vf(e_u)\vf(e_{uv})=\vf(e_{uv})\vf(e_v)$ for all $u<v$ from $X\sm\{x\}$ by \cref{vf(I(X-minus-x))-sst-I(X-minus-y),zp-on-the-basis}. Hence, $\af_u=\af_v$. Since $X$ is connected, we are done.
	\end{proof}
	
	Replacing $\vf$ by $-\vf$, if necessary, we may assume that
	\begin{align}\label{vf(dl)=dl}
		\vf(\dl)=\dl.
	\end{align}
	Then, for all $z\in X$, we have
	\begin{align}\label{vf(e_z)=e_z}
		\vf(e_z)=e_z.
	\end{align}
	In particular, $c=\vf(e_{X\sm\{x\}})=e_{X\sm\{y\}}$, so $\vf=\psi$ on $I(X\sm\{x\},F)$. We will now show that these conditions guarantee that $\vf$ is a multiplicative automorphism \cref{M_sg(e_xy)=sg(x_y)e_xy} of $I(X,F)$.
	
	\begin{lem}\label{vf-mult-auto}
		Let $X$ be connected. Then $\vf$ is a multiplicative automorphism of $I(X,F)$.
	\end{lem}
	\begin{proof}
		Consider $\sg:X^2_<\to F^*$ from \cref{vf(e_uv)-multiple-e_uv} and extend it to $\sg:X^2_\le\to F^*$ by means of $\sg(z,z)=1$ for all $z\in X$. Observe that \cref{vf(e_uv)=sg(u_v)e_uv} now holds for all $u\le v$ thanks to \cref{vf(e_z)=e_z}. Thus, it only remains to prove that $\sigma$ satisfies \cref{sg(x_y)sg(y_z)=sg(x_z)}. Take arbitrary $u\leq v \leq w$. We need to show that $\sg(u,v)\sg(v,w)=\sg(u,w)$. Since $\sg(z,z)=1$ for all $z\in X$, it suffices to assume $u<v<w$. Consider the following cases.
		
		{\it Case 1.} $x\notin \{u,v,w\}$. Then $\sigma(u,v)\sigma(v,w)=\sigma(u,w)$ because $\vf=\psi$ on $I(X\setminus\{x\},F)$ and $\psi(e_{uw})=\psi(e_{uv})\psi(e_{vw})$.
		
		{\it Case 2.} $x=u$. Since
		$(e_x+e_{xv}-e_{xw})(\delta-e_{xv}+e_{vw})=e_x$, by \cref{vf(dl)=dl,vf(e_uv)=sg(u_v)e_uv} we have
		\begin{align*}
			(e_x+\sg(x,v)e_{xv}-\sg(x,w)e_{xw})(\dl-\sg(x,v)e_{xv}+\sg(v,w)e_{vw})=e_x,   
		\end{align*}
		which gives $\sigma(x,v)\sigma(v,w)=\sg(x,w)$. 
		
		{\it Case 3.} $x=w$. This case is proved similarly to Case 2 using the product $(\delta-e_{vx}+e_{uv})(e_x+e_{vx}-e_{ux})=e_x$.
		
		{\it Case 4.} $x=v$. The result follows by applying $\vf$ to the product $(e_x+e_{ux}+e_{xw}+e_{uw}-e_u) (\delta-e_{xw}+e_{ux}-e_u)=e_x$.
	\end{proof}
	
	We can finally prove the main result of our work.
	\begin{thrm}\label{main-result}
		Let $X$ be a finite connected poset and $x,y\in X$. Let $\e,\eta\in I(X,F)$ be primitive idempotents such that $\e_D=e_x$ and $\eta_D=e_y$. There exists a bijective linear map $\vf:I(X,F)\to I(X,F)$ preserving products equal to $\e$ and $\eta$ if and only if there exists an automorphism of $X$ mapping $x$ to $y$, in which case $\vf$ is either an automorphism of $I(X,F)$ or the negative of an automorphism of $I(X,F)$. 
	\end{thrm}
	\begin{proof}
		Assume that $\vf:I(X,F)\to I(X,F)$ preserves products equal to $\e$ and $\eta$. By \cref{lem conjugate T,lambda-auto,vf-mult-auto}, up to an inner automorphism, $\vf$ is of the form $\pm\wht\lb\circ M_\sg$, where $\lb$ is an automorphism of $X$ with $\lb(x)=y$ and $M_\sg$ is a multiplicative automorphism of $I(X,F)$.
		
		Conversely, if $\lb$ is an automorphism of $X$ with $\lb(x)=y$, then $\wht\lb\in\Aut(I(X,F))$ preserves products equal to $e_x$ and $e_y$. Composing $\wht\lb$ with appropriate inner automorphisms of $I(X,F)$, we get $\vf\in\Aut(I(X,F))$ which preserves products equal to $\e$ and $\eta$.
	\end{proof}
	
	\begin{cor}\label{result-for-T_n(K)}
		Let $\e,\eta\in T_n(F)$ be primitive idempotents. There exists a bijective linear map $\vf:T_n(F)\to T_n(F)$ preserving products equal to $\e$ and $\eta$ if and only if $\e$ and $\eta$ have the same diagonal, in which case $\vf$ is either an automorphism or the negative of an automorphism of $T_n(F)$.
	\end{cor}
	\begin{proof}
		For, $T_n(F)\cong I(X,F)$, where $X$ is a chain of length $n-1$. Then, $\e$ and $\eta$, seen as elements of $I(X,F)$, have diagonals $e_x$ and $e_y$ for some $x,y\in X$. Since $\Aut(X)=\{\id\}$, the result follows.
	\end{proof}

	\section*{Acknowledgements}
	The first author was partially supported by Junta de Andalucía (grants FQM375 and PY20\texttt{\char`_}00255).
	The second author was partially supported by CNPq (process 404649/2018-1). We thank the referee for pointing out inaccuracies throughout the text and suggestions that improved its clarity.

	\bibliography{bibl}{}
	\bibliographystyle{acm}

\end{document}